\documentclass{amsart}
\usepackage{amsmath,amssymb,amsthm,verbatim}
\usepackage{tikz}
\newtheorem{theorem}{Theorem}
\newtheorem{proposition}{Proposition}
\newtheorem{lemma}{Lemma}
\newtheorem{corollary}{Corollary}
\theoremstyle{definition}
\newtheorem{definition}{Definition}
\begin{document}
\title[Generators for Tensor Product Components]{Generators for Decompositions of Tensor Products of Indecomposable Modules}
\author{Michael~J.~J.~Barry}
\address{15 River Street Unit 205\\
Boston, MA 02108}
\email{mbarry@allegheny.edu}


\begin{abstract}
Let $F$ be a field of non-zero characteristic $p$, let $G$ be a cyclic group of order $q =p^a$ for some positive integer $a$, and let $U$ and $W$ be indecomposable $F G$-modules.   We identify a generator for each of the indecomposable components of $U \otimes W$ in terms of a particular $F$-basis of $U \otimes W$, extending previous work in which each component occurred with multiplicity one.   An anonymous reviewer helped us realize that our results are really just a translation of previous results of Iiwa and Iwamatsu to our situation.
\end{abstract}

\maketitle

\section{Introduction}
Let $F$ be a field of non-zero characteristic $p$ and $G$ be a cyclic group of order $q =p^a$ for some positive integer $a$.  It is well-known every indecomposable $F G$-module has $F$-dimension at most $q$, that  there are indecomposable $F G$-modules of every dimension between $1$ and $q$ with any two of the same dimension isomorphic, and that such modules are cyclic and uniserial~\cite[pp. 24--25]{Alperin}.  Let $\{V_1,V_2,\dots,V_q\}$ be a complete set of representatives of the isomorphism classes of indecomposable $F G$-modules with $\dim V_i=i$.
Many authors have investigated the decomposition of $V_m \otimes V_n$ where $m \leq n \leq q$ --- for example, in order of appearance see~\cite{Green1962,Bhama1964, Ralley1969,Renaud1979,McFaul1980,Norman1995,Hou2003,IandI2009,B2011_0}.  From the work of these authors, we know that
$V_m \otimes V_n$ is isomorphic to a direct sum  $V_{\lambda_1} \oplus V_{\lambda_2} \oplus \dots \oplus V_{\lambda_m}$  of $m$ indecomposables where  $\lambda_1\geq \lambda_2 \geq \dots \geq \lambda_m>0$ but that the dimensions $\lambda_i$ depend on the characteristic $p$.

Fix a generator $g$ of $G$. There is a basis $\{u_1,\dots,u_m\}$ of $V_m$ on which the action of $g$ is given by $g u_1=u_1$ and $g u_i=u_{i-1}+u_i$ when $i>1$.  Similarly there is a basis $\{w_1,\dots,w_n\}$ of $V_n$ on which the action of $g$ is given by $g w_1=w_1$ and $g w_i=w_{i-1}+w_i$ when $i>1$.  (We could if we wish take $V_m=\langle w_1,\dots,w_m \rangle$.)  Thus the matrix of $g$ with respect to this basis of $V_n$ is $J_n$, the full $n \times n$ Jordan matrix with eigenvalue $1$.  Clearly $\{u_i \otimes w_j \mid 1 \leq i \leq m,1 \leq j \leq n\}$ is basis for $V_m \otimes V_n$ over $F$.  But $\mathcal{B}=\{u_i \otimes g^{n-i}w_j \mid 1 \leq i \leq m,1 \leq j \leq n\}$ is also a basis for $V_m \otimes V_n$ that is nicer to work with.    For each $i \in [m]=\{1,2,\dots,m\}$, we will identify $y_i$, expressed in terms of $\mathcal{B}$, such that $F Gy_i \cong V_{\lambda_i}$ and $V_m \otimes V_n$ is an internal direct sum of the indecomposable  modules $F G y_i$.  We describe the generators $y_i$ in terms of the dimensions $\lambda_i$ and the inverses of linear transformations $S_{b_i}$ whose matrices with respect to particular bases have binomial entries.

Crucial to our presentation is a result of Iima and Iwamatsu(I\&I)~\cite{IandI2009} on when certain integer-valued determinants are non-zero modulo $p$.   But if we had read~\cite{IandI2009} more carefully, we would have seen that I\&I had established essentially the same results as ours in their parallel but equivalent situation.  They wish to find the Jordan canonical form of $J_{1,m} \otimes J_{1,n}$, where $J_{\alpha,s}$ is the Jordan block with eigenvalue $\alpha$ and size $s$.  They set $R=F[x,y]/(x^m,y^n)$ and $\theta=x+y$.  Their problem is reduced to finding the indecomposable decomposition of $R$ as an $F[\theta]$-module.  They prove that $R=\bigoplus_{i=0}^{m-1}F[\theta]\omega_i$, where $\omega_i$ is an $i$-th degree homogeneous element.  Their $F[\theta]\omega_i$ corresponds to our $V_{\lambda_{i+1}}$,  their $\theta$ to our $g-1$, and their $\omega_i$  to our $y_{i+1}$. 

Though our work is mostly an unwitting translation of previous work,  a recent result allows our identification of $y_{i+1}$ to be much more precise that I\&I's description of $\omega_i$, as we will explain in the next paragraph.  In addition, we believe that there is value in setting out the general solution in our setting since  it  completes and improves on earlier work on special cases ~\cite{Barry2011,B2017} in that setting. 

 In~\cite{IandI2009}, when $\omega_i$ corresponds to a \emph{leading module} of $R$,  it is denoted $\kappa_{(i)}$ and defined as a nonzero element in the  kernel of a certain linear transformation which we will denote as $L_1$.   In this case,  we give $y_{i+1}$ as $L_2(x_{i+1})$,  where $x_{i+1}$ is a specific element of $V_m \oplus V_n$ and $L_2$ is a linear transformation whose action with respect to certain bases is explicitly given.  So we can compute $y_{i+1}$ directly whereas $I\&I$ must solve for the kernel of $L_1$.  I\&I argue that $\kappa_{(i)}$ can be chosen freely because of the uniqueness of the indecomposable decomposition.  But in our set-up we will see that $\kappa_{(i)}$ is unique up to a non-zero multiple  because $\ker L_1$ is one-dimensional. The explicit action of $L_2$ that we give is possible thanks to recent work of Nordenstam and Young~\cite{NordYoung}, which was brought to our attention by the reviewer.

We should point out that Norman, in~\cite{Norman1995,Norman2008}, studied the equivalent problem of the nilpotent linear transformation $S$ on the $m \times n$ matrices $X$ over $F$, where $(X)S=(J_m-I_m)^T X +X (J_n-I_n)$ and gave a recursive specification for a Jordan basis of $S$ in terms of the base-$p$ expansions of $m$ and $n$.  The identification we present here is easier to compute.  

We now describe the organization of the paper.  In  Section~\ref{prelim} we give just enough background to allow us to state and comment on our main theorem in Section~\ref{Statement}.   Then in Section~\ref{AuxResults}, we assemble some auxiliary results needed for the proof of our main theorem given in Section~\ref{Proofs}.  At various points we will comment on how our results align with those of I\&I.  We conclude by computing an example in Section~\ref{Example}.

\section{Preliminaries}\label{prelim}
We record some definitions and results that we will need in order to state our main result.
\begin{definition}\label{def1}
\begin{enumerate}
\item For each $(i,j)$ with $1 \leq i \leq m$ and $1 \leq j \leq n$, define $v_{i,j}=u_i \otimes g^{n-i}w_j$.
\item For every $k \in [m+n-1]$, let $D_k=\langle v_{i,j} \mid i+j=k+1 \rangle$.    Also for convenience define $D_k=\{0\}$ if $k<1$.
\item For each $i \in [m]$, let $x_i=\sum_{j=1}^i (-1)^{j-1}v_{j,i+1-j} \in D_i$.
\end{enumerate}
\end{definition}
\begin{lemma}\label{Prelim}
\begin{enumerate}
\item The set $\mathcal{B}=\{v_{i,j}=u_i \otimes g^{n-i}w_j \mid 1 \leq i \leq m,1 \leq j \leq n\}$ is a basis for $V_m \otimes V_n$ over $F$, and $(g-1)(v_{i,j})=v_{i-1,j}+v_{i,j-1}$ where we understand that $v_{k,\ell}=0$ if $k<1$ or $\ell<1$.
\item  Then $V_m \otimes V_n=\bigoplus_{i=1}^{m+n-1} D_i$ as a direct sum of vector spaces, $(g-1)(D_k) \subseteq D_{k-1}$, and
\[(g-1)^r(v_{i,j})=\sum_{k=0}^r \binom{r}{k} v_{i+k-r,j-k}.\]
\item  The set $\{x_1,x_2,\dots,x_m\}$ is linearly independent over $F$ and $(g-1)x_i=0$ for every $i$.
\end{enumerate}
\end{lemma}
Here our $D_{m+n-i}$, $g-1$, and $v_{i,j}$ correspond to I\&I's $R_{i-1}$, $\theta$, and $x^{m-i}y^{n-j}$, respectively. 
\begin{proof}
(1) See~\cite[Lemma 1]{B2017} or~\cite[Lemma 4]{Norman1993}.

(2) Follows from part (1) and the binomial formula.

(3) See~\cite[Lemma 2]{Norman1995}.
\end{proof}

For each integer $i \in [2,m+n-1]$, denote $(g-1)_{|D_i} \colon D_i \to D_{i-1}$ by $T_i$.
\begin{lemma}\label{InjIsoSur}
The map $T_i$ is injective when $n+1 \leq i \leq m+n-1$,  bijective when $m+1 \leq i \leq n$, and surjective with kernel $\langle x_i \rangle$ when $2 \leq i \leq m$.
\end{lemma}
\begin{proof}
By~\cite[Lemma 2]{Norman1995}, $\{x_1,\dots,x_m\}$ is a basis for the kernel of the action of $(g-1)$ on $V_m \otimes V_n$.  The result follows immediately then since $D_i \cap \langle x_1, \dots, x_m \rangle=\langle x_i \rangle$ when $1 \leq i \leq m$ and $D_i \cap \langle x_1, \dots, x_m \rangle=\{0\}$ when $m < i \leq m+n-1$.
\end{proof}

\begin{corollary}
The kernel of the action of $g-1$ on $V_m \otimes V_n$ is $\langle x_1,\dots, x_m \rangle$.
\end{corollary}

Next we define a particular ordered basis $\mathcal{B}_i$ for each $D_i$, ($1 \leq i \leq m+n-1$).
\begin{definition}
\[\mathcal{B}_i=
\begin{cases}
(v_{1,i},v_{2,i-1},\dots,v_{i,1}), & \text{$1 \leq i \leq m$,}\\
(v_{1,i},v_{2,i-1},\dots,v_{m,i+1-m}), & \text{$m+1 \leq i \leq n$,}\\
(v_{i+1-n,n},v_{i+2-n,n-1},\dots,v_{m,i+1-m}), & \text{$n+1 \leq i \leq m+n-1$.}
\end{cases}
\]
\end{definition}

\section{Statement of Main Theorem}\label{Statement}
\begin{theorem}\label{MainTheorem}
Suppose that $m$ and $n$ are positive integers with $m \leq n$ and that
\[V_m \otimes V_n \cong V_{\lambda_1} \oplus V_{\lambda_2} \oplus \dots \oplus V_{\lambda_m}
\]
where $\lambda_1 \geq \lambda_2 \geq \dots \geq \lambda_m>0$. For $i\in [m]$, let $a_i+1=\min\{k \mid \lambda_k=\lambda_i\}$ and $b_i=\max\{k \mid \lambda_k=\lambda_i\}$.  Then $\lambda_{k}=m+n-a_i-b_i=\lambda_i$ for every integer $k \in [a_i+1,b_i]$ and $(g-1)^{m+n-2b_i} \colon D_{m+n-b_i} \to D_{b_i}$ is a bijection.  Denote $(g-1)^{m+n-2b_i}_{|D_{m+n-b_i}}$ by $S_{b_i}$.

We consider three cases.
\begin{enumerate}
\item $a_i+1=b_i$: Then $i=b_i$ and $\lambda_i=m+n-2i+1=m+n-2 b_i+1$.  Let $y_{b_i}=S_{b_i}^{-1}(x_{b_i})$.   The action of $S_{b_i}^{-1}$ is given in Proposition~\ref{Adjoint}.
\item $a_i+1<b_i$ and $i=a_i+1$:  Then $(g-1)^{\lambda_i-1} \colon D_{m+n-a_i-1} \to D_{b_i}$ is the composition $S_{b_i} \circ T$ where $T=(g-1)^{b_i-a_i-1}_{|D_{m+n-a_i-1}}$ is an injection.  Then $S_{b_i}^{-1}(x_{b_i}) \in \text{Im }T$.  There is a linear transformation $U \colon D_{m+n-b_i} \to D_{m+n-a_i-1}$, described in Lemma~\ref{TInverse},  such that $U_{|\text{Im }T}$ is the inverse of $T$.  Let $y_{a_i+1}=U(S_{b_i}^{-1}(x_{b_i}))\in D_{m+n-a_i-1}$. 
\item $a_i+1<b_i$ and $a_i+2 \leq i \leq b_i$:  Suppose that the expression for $y_{a_i+1}$ given in Case 2 is $y_{a_i+1}=\sum_{j=m-a_i}^m\alpha_{j} v_{j,m+n-a_i-j}$. Define
\[y_i=(-1)^{i-a_i-1}\sum_{j=m-i+1}^{m+a_i-i+1}\alpha_{j+i-a_i-1}v_{j,m+n+1-i-j} \in D_{m+n-i}.
\]
\end{enumerate}
Then, $(g-1)^{\lambda_i-1}(y_i)=x_{m+n-i-\lambda_i+1}$ and $F G y_i \cong V_{\lambda_i}$ for every $i \in [m]$,  and
\[V_m \otimes V_n =\bigoplus_{i=1}^m FG y_i.\]
\end{theorem}

\textbf{Comments}

\begin{enumerate}
\item Theorem~\ref{MainTheorem} assumes that the dimensions $\lambda_i$ are known.  This is not a drawback since there are lots of algorithms for computing the $\lambda_i$, for example, see~\cite{IandI2009} or the recent~\cite{Barry2021}.
\item When $i=a_i+1$, which occurs in Cases 1 and 2 of the Theorem, there is a unique $y_i \in D_{m+n-i}$ such that $(g-1)^{\lambda_i-1}(y_i)=x_{m+n-i-\lambda_i+1}$.  This is true because $(g-1)^{\lambda_i-1}$ is injective on $D_{m+n-i}$.  The linear transformation $L_2$ referred to in the Introduction is $U \circ S_{b_i}^{-1}$,  where $U=I$ in case (1). 
\item In Case 3,  where $a_i+1<i \leq b_i$, the equation $(g-1)^{\lambda_i-1}(y)=x_{m+n-i-\lambda_i+1}$ does not have a unique solution for $y \in D_{m+n-i}$, but we use the coefficients of $y_{a_i+1}$ to to single out a particular solution $y_i$.  Another way to express Case 3 is that if $y_{a_i+1}$ has coefficients 
\[(\alpha_{m-a_i},\alpha_{m-a_i+1},\dots,\alpha_{m})\]with respect to the ordered basis
\[\mathcal{B}_{m+n-a_i-1}=(v_{m-a_i,n},v_{m-a_i+1,n-1},\dots,v_{m,n-a_i})\text{ of } D_{m+n-a_i-1},\] 
 then $y_{i}$  has coefficients
 \[(-1)^{i-a_i-1}(\alpha_{m-a_i},\alpha_{m-a_i+1},\dots,\alpha_{m},\underbrace{0,0,\dots,0}_{i-a_i-1})\] 
 with respect to the ordered basis
\[\mathcal{B}_{m+n-i}=(v_{m-i+1,n},v_{m_i+2,n-1},\dots,v_{m,n-i+1})\text{ of }D_{m+n-i}.\]
\item The modules $V_{a_i+1}$ correspond to the \emph{leading modules} of I\&I, while Lemma 2.2.6 and Lemma 2.27 of I\&I correspond to case (1) and case (2), respectively,  of the above Theorem.   I\&I's $\kappa_{(a_i)}$ is in our language a non-zero element of the kernel of $(g-1)^{\lambda_{a_i+1}}$ acting on $D_{m+n-a_i-1}$.  Since $(g-1)^{\lambda_{a_i+1}-1}$ maps $D_{m+n-a_i-1}$ injectively into $D_{m+n-a_i-\lambda_{a_i+1}}$ and the kernel of $g-1$ acting on $D_{m+n-a_i-\lambda_{a_i+1}}$ is $\langle x_{m+n-a_i-\lambda_{a_i+1}} \rangle$,  it follows that the kernel of $(g-1)^{\lambda_{a_i+1}}$ acting on $D_{m+n-a_i-1}$ is $\langle y_{a_i+1} \rangle$, that is, this kernel is one-dimensional.
\end{enumerate}

\section{Auxiliary Results}\label{AuxResults}

\begin{lemma}\label{LTMatrix}
Let $r$ and $s$ be positive integers satisfying $1 \leq r <s \leq m+n-1$.  Let $m_r=\max\{1,r+1-n\}$, $m_s=\max\{1,s+1-n\}$, $M_r=\min\{r,m\}$, and $M_s=\min\{s,m\}$.  Thus $m_r \leq m_s$, $M_r \leq M_s$, 
\[\mathcal{B}_r=(v_{m_r,r+1-m_r},v_{m_r+1,r-m_r},\dots,v_{M_r,r+1-M_r}),\]
and
\[\mathcal{B}_s=(v_{m_s,s+1-m_s},v_{m_s+1,s-m_s},\dots,v_{M_s,s+1-M_s}).\]   Then the $(M_r-m_r+1)\times (M_s-m_s+1)$ matrix of $(g-1)^{s-r} \colon D_s \to D_r$ with respect to $\mathcal{B}_s$ and $\mathcal{B}_r$ is
\[
\begin{pmatrix}
\binom{s-r}{m_s-m_r} & \binom{s-r}{m_s-m_r+1} & \dots & \binom{s-r}{M_s-m_r}\\[4pt]
\binom{s-r}{m_s-m_r-1} & \binom{s-r}{m_s-m_r} & \dots & \binom{s-r}{M_s-m_r-1}\\[4pt]
\vdots & \vdots & \ddots & \vdots \\[4pt]
\binom{s-r}{m_s-M_r} & \binom{s-r}{m_s-M_r+1} & \dots  & \binom{s-r}{M_s-M_r} \\[4pt]
\end{pmatrix},
\] 
that is, the $(i,j)$th entry is $\binom{s-r}{m_s-m_r+j-i}$.
\end{lemma}
\begin{proof}
The result follows easily from Lemma~\ref{Prelim} (2).
\end{proof}
For an integer $k \in [1,m]$, denote the matrix  of $S_k=(g-1)^{n+m-2k} \colon D_{m+n-k} \to D_k$ with respect to $\mathcal{B}_{m+n-k}$ and $\mathcal{B}_k$ by $A_k$.  When $A_k$ is invertible, we apply a result of Nordenstam and Young to compute $A_k^{-1}$.

\begin{proposition}\label{Adjoint}
Let $M$ be the $k \times k$ matrix with $(i,j)$-entry $\binom{a}{b+j-i} \in F$ and suppose that $M$ is invertible.  Let $d_k=\prod_{i=0}^{k-1} \binom{a+i}{b}/\binom{b+i}{b} \in \mathbb{Q}$, and for each $(i,j)$ with $1 \leq i,j \leq k$, define
\[
z_{i,j}=\frac{d_k}{\binom{a+k-1}{b+i-1}}\sum_{\ell=1}^j(-1)^{\ell+j}\binom{a+k-1}{\ell-1}\binom{a+j-\ell-1}{j-\ell}\prod_{r=1,r \neq i}^k \frac{\ell-b-r}{i-r}
\in \mathbb{Q}.
\]
Then $d_k \in \mathbb{Z}$ and $z_{i,j} \in \mathbb{Z}$ for every $(i,j)$.  Define the $k \times k$ matrix $N$ whose $(i,j)$-entry is $\phi(z_{i,j})$, where $\phi$ is the natural ring homomorphism of $\mathbb{Z}$ into $F$.  Then $\det M=\phi(d_k)$ and $M^{-1}=(\det M)^{-1} N$.
\end{proposition}
\begin{proof}
Let $M_\mathbb{Q}$ be the $k \times k$ matrix with $(i,j)$-entry $\binom{a}{b+j-i} \in \mathbb{Q}$.  Then $\det M_\mathbb{Q} \in \mathbb{Z}$, actually equal to $d_k$ by a result of Roberts~\cite{Roberts}, and $adj(M_\mathbb{Q})$, the adjoint of $M_\mathbb{Q}$ is a matrix of integers, satisfying
$M_\mathbb{Q} adj(M_\mathbb{Q})= d_k I_k=adj(M_\mathbb{Q}) M_\mathbb{Q}$.
Replacing  $(A,B_j,n)$ by $(a,b+j,k)$ in~\cite[Theorem 1]{NordYoung}, we see that $z_{i,j}$ is the $(i,j)$-entry of $adj(M_\mathbb{Q})$.  Now apply $\phi$ to conclude that $M N=\phi(d_k) I_k=(\det M)I_k=N M$ and $M^{-1}=(\det M)^{-1} N$.
\end{proof}

The matrix $A_k$ is just $M$ in Proposition~\ref{Adjoint} with parameters $a=m+n-2k$ and $b=m-k$.

We repurpose a lemma of Hou~\cite[Lemma 2.1]{Hou2003}.

\begin{lemma}\label{Hou}
Consider
\begin{multline*}
D_{m+n-1}\overset{f_1} {\longrightarrow} D_{m+n-2} \overset{f_2} {\longrightarrow} \dots \overset{f_{m-2}}{\longrightarrow} D_{n+1}\overset{f_{m-1}} {\longrightarrow} D_n\\
 \overset{f_m}{\longrightarrow}D_m \overset{h_m}{\longrightarrow}D_{m-1} \overset{h_{m-1}}{\longrightarrow} \dots\overset{h_2}{\longrightarrow} D_1 \overset{h_1}{\longrightarrow}\{0\}
\end{multline*}
where $f_i=T_{m+n-i}$ when $1 \leq i \leq m-1$, $f_m=(g-1)^{n-m}_{|D_n}$, and $h_i=T_i$.  Then the $f_i$ are injective with $f_m$ bijective and the $h_i$ are surjective.  Then there is a permutation $\pi$ of $[m]$ and $z_i \in D_{m+n-i}$ ($1 \leq i \leq m$) such that
\[\langle h_{\pi(i)+1}h_{\pi(i)+2}\dots h_m f_{m} \dots f_i(z_i)\rangle=\text{ker } h_{\pi(i)}=\langle x_{\pi(i)}\rangle\]
where $h_{\pi(i)+1}h_{\pi(i)+2}\dots h_m f_{m} \dots f_i$ is interpreted as $f_m$ if $\pi(m)=m$.
\end{lemma}

\begin{proof}
The injectivity/surjectivity/bijectivity of the maps follows from Lemma~\ref{InjIsoSur}.  Apart from slight changes in notation, the proof follows from that of Hou~\cite[Lemma 2.1]{Hou2003}.
\end{proof}

\begin{corollary}\label{CorHou}
Suppose that
\[V_m \otimes V_n \cong V_{\lambda_1} \oplus \dots \oplus V_{\lambda_m}
\]
where $\lambda_1 \geq \dots \geq \lambda_m>0$. For each $i \in [m]$, there exists $y_i \in D_{m+n-i}$ such that $(g-1)^{\lambda_i-1}(y_i)=x_{m+n-i-\lambda_i+1}$.
\end{corollary}
\begin{proof}
By Lemma~\ref{Hou}, there exists a permutation $\pi$ of $[m]$ and $y_i \in D_{m+n-i}$, $1 \leq i \leq m$, such that $(g-1)^{m+n-i-\pi(i)}(y_i)= x_{\pi(i)}$.  Then $F G y_i$ is an indecomposable submodule of $V_m \otimes V_n$ of dimension $m+n-i-\pi(i)+1$ over $F$, hence isomorphic to $V_{m+n-i-\pi(i)+1}$, with one-dimensional socle  $\langle x_{\pi(i)} \rangle$. Also $\sum_{i=1}^m F G y_i$ is a direct sum since the socles are linearly independent by Lemma~\ref{Prelim}. Lastly, since $\sum_{i=1}^m \dim_F F G y_i=\sum_{i=1}^m (m+n-i-\pi(i)+1)=m n$,
\[V_m \otimes V_n = \bigoplus_{i=1}^m F G y_i.\]
Now $m+n+1-i-\pi(i)$ is non-increasing by the comment after~\cite[Theorem 2.2]{Hou2003} which references~\cite[Corollary 5]{Norman1995}.  Therefore by Krull-Schmidt, $\lambda_i=m+n+1-i-\pi(i)$.  Hence for each integer $i$ in $[1,m]$,
\[(g-1)^{\lambda_i-1}(y_i)=x_{\pi(i)}=x_{m+n+1-i-\lambda_i}.\]
\end{proof}
Corollary~\ref{CorHou} corresponds to Theorem 2.2.2 of I\&I.

Example 1: When the characteristic of $F$ is $5$, $V_6 \otimes V_9 \cong V_{14} \oplus V_{10}\oplus V_{10} \oplus V_{10} \oplus V_6 \oplus V_4$.  By Corollary~\ref{CorHou}, for each integer $i$ in $[1,6]$, there is a $y_i \in D_{15-i}$ such $(g-1)^{13}(y_1)=x_1$, $(g-1)^9(y_2)=x_4$, $(g-1)^9(y_3)=x_3$, $(g-1)^9(y_4)=x_2$, $(g-1)^5(y_5)=x_5$, and $(g-1)^3(y_6)=x_6$.

\begin{lemma}\label{BinomSum}
Let $s$, $i$, and $j$ be positive integers with $i \leq j$.  Then
\[
\sum_{k=i}^j\binom{s}{k-i}(-1)^{j-k}\binom{s-1+j-k}{s-1}=\delta_{i,j}.
\]
\end{lemma}
\begin{proof}
Formula (5.25) of~\cite{GKP} is
\[\sum_{k \leq \ell} \binom{\ell-k}{r}\binom{s}{k-t}(-1)^k=(-1)^{\ell+r} \binom{s-r-1}{\ell-r-t},
\]
where $\ell$, $r$, and $t$ are nonnegative integers.  Taking $\ell=m-1+j$, $r=m-1$, $s=m$, and $t=i$ gives
\begin{align*}
\sum_{k \leq m-1+j}\binom{m-1+j-k}{m-1} &\binom{m}{k-i}(-1)^{k}\\&=(-1)^{2 m-2+j}\binom{m-(m-1)-1}{m-1+j-(m-1)-i}\\
&=(-1)^j \binom{0}{j-i}\\
&=(-1)^j \delta_{i,j}.
\end{align*}
Now $\{k \mid k \leq m-1+j, \binom{m-1+j-k}{m-1} \binom{m}{k-i}\neq 0\}=[i,j]$.  Therefore
\begin{align*}
\sum_{k=i}^j\binom{m}{k-i}(-1)^{j-k}\binom{m-1+j-k}{m-1}&=(-1)^j\sum_{k=i}^j\binom{m}{k-i}(-1)^{k}\binom{m-1+j-k}{m-1}\\
&=\delta_{i,j}.
\end{align*}
\end{proof}

\begin{lemma}\label{InverseMatrices}
Let $r$ and $s$ be positive integers.  Define the $r \times r$ matrices $A(r,s)$ with $(i,j)$th entry $a_{i,j}=\binom{s}{j-i}$ and $B(r,s)$ with $(i,j)$th entry $b_{i,j}=(-1)^{j-i}\binom{s-1+j-i}{s-1}$.  Then $A(r,s)$ and $B(r,s)$ are upper-triangular unipotent matrices and \\$A(r,s)^{-1}=B(r,s)$.
\end{lemma}
\begin{proof}
This is a special case of~\cite[Theorem 1]{NordYoung}.
The product $A(r,s)B(r,s)$ is upper-triangular unipotent.
When $j \geq i$, the $(i,j)$th entry of $A(r,s)B(r,s)$ is
\[\sum_{k=i}^j\binom{s}{k-i}(-1)^{j-k}\binom{s-1+j-k}{s-1},
\]
which equals $\delta_{i,j}$ by Lemma~\ref{BinomSum}.  Thus $A(r,s)^{-1}=B(r,s)$.
\end{proof}

\begin{lemma}\label{TInverse}
Suppose that $a$ and $b$ are positive integers with $a+1<b \leq m$.  Recall that  $\mathcal{B}_{m+n-a-1}=(v_{m-a,n},v_{m-a+1,n-1},\dots,v_{m,n-a})$ and \\$\mathcal{B}_{m+n-b}=(v_{m-b+1,n},v_{m-b+2,n-1}, \dots, v_{m,n-b+1})$.  Denote by $T$ the linear transformation $(g-1)^{b-a-1} \colon D_{m+n-a-1} \to D_{m+n-b}$.  Define the linear transformation $U \colon D_{m+n-b} \to D_{m+n-a-1}$ by
\begin{align*}
U&\left(\sum_{j=m-b+1}^m \alpha_{j} v_{j,m+n-b+1-j} \right)\\
&=\sum_{k=m-a}^m \left(\sum_{i=k}^m (-1)^{i-k}\binom{i-k+b-a-2}{b-a-2} \alpha_{i}\right)v_{k,m+n-a-k}.
\end{align*}
(So $U(v_{m-b+1,n})=\dots=U(v_{m-a-1,n-b+a+2})=0$.)
Then $T \circ U_{|\text{Im }T}=I_{\text{Im }T}$ and $U \circ T=I_{D_{m+n-a-1}}$.
\end{lemma}
\begin{proof}
Since $(T(v_{m-a,n}),T(v_{m-a+1,n-1}),\dots,T(v_{m,n-a}))$ is a basis for $\text{Im }T$, it suffices to show that $U(T(v_{j,m+n-a-j}))=v_{j,m+n-a-j}$ when $m-a \leq j \leq m$.
By Lemma~\ref{LTMatrix}, the $b \times (a+1)$ matrix of $T$ with respect to $\mathcal{B}_{m+n-a-1} $ and $\mathcal{B}_{m+n-b}$ is
\[
M(T)=\begin{pmatrix}
\binom{b-a-1}{b-a-1} & 0  & \dots & 0\\[4pt]
\binom{b-a-1}{b-a-2} & \binom{b-a-1}{b-a-1}  & \dots & 0\\[4pt]
\vdots & \vdots &  \ddots & \vdots\\[4pt]
\binom{b-a-1}{0} & \binom{b-a-1}{1} & \dots & \binom{b-a-1}{a}\\[4pt]
0 & \binom{b-a-1}{0} & \dots & \binom{b-a-1}{a-1}\\[4pt]
\vdots & \vdots  & \ddots & \vdots\\[4pt]
0 & 0  & \dots & \binom{b-a-1}{0}
\end{pmatrix}.
\]
Thus 
\[M(T)=
\begin{pmatrix}
T_1\\
A(a+1,b-a-1)
\end{pmatrix}\]
where $T_1$ has size $(b-a-1) \times (a+1)$ and $A(a+1,b-a-1)$ is defined in Lemma~\ref{InverseMatrices}.

Now

\begin{align*}
U(&v_{m-a+j,n-b+a+1-j})\\
&=\sum_{k=m-a}^m \left(\sum_{i=k}^m (-1)^{i-k}\binom{i-k+b-a-2}{b-a-2} \delta_{i,m-a+j}\right)v_{k,m+n-a-k}\\
&=\sum_{k=m-a}^{m-a+j} \left(\sum_{i=k}^m (-1)^{i-k}\binom{i-k+b-a-2}{b-a-2} \delta_{i,m-a+j}\right)v_{k,m+n-a-k}\\
&=\sum_{k=m-a}^{m-a+j} (-1)^{m-a+j-k}\binom{m-a+j-k+b-a-2}{b-a-2}v_{k,m+n-a-k}\\
&=\sum_{i=0}^j(-1)^{j-i}\binom{b-a-2+j-i}{b-a-2}v_{m-a+i,n-b+a+1-i}.
\end{align*}

Hence the $(a+1) \times b$  matrix of $U$ with respect to $\mathcal{B}_{m+n-b}$ and $\mathcal{B}_{m+n-a-1}$, which has its leftmost $b-a-1$ columns all $0$, is
\[M(U)=
\begin{pmatrix}
0 & \dots & 0 & \binom{b-a-2}{b-a-2} & -\binom{b-a-1}{b-a-2} & \binom{b-a}{b-a-2} & \dots  & (-1)^a \binom{b-2}{b-a-2}\\[4pt]
0 & \dots & 0 & 0 & \binom{b-a-2}{b-a-2} & -\binom{b-a-1}{b-a-2} & \dots & (-1)^{a-1} \binom{b-3}{b-a-2}\\[4pt]
0 & \dots & 0 & 0 & 0 & \binom{b-a-2}{b-a-2}& \dots & (-1)^{a-2}\binom{b-4}{b-a-2}\\[4pt]
\vdots & \ddots & \vdots & \vdots &  \vdots & \vdots & \ddots & \vdots\\[4pt]
0 & \dots & 0 & \vdots &  \vdots & \vdots & \dots & \binom{b-a-2}{b-a-2}
\end{pmatrix}.
\]
The $(i,j)$th entry is $0$ if $j<b-a$ or if $j-i\leq b-a$, and $(-1)^{j-b+a-i+1}\binom{j-i-1}{b-a-2}$ otherwise.  Thus the matrix of $U$ is 
\[\left(0_{a+1,b-a-1} \mid B(a+1,b-a-1)\right),
\]
where $B(a+1,b-a-1)$ is defined in Lemma~\ref{InverseMatrices}.
Then
\begin{align*}
M(U)M(T)&=0_{a+1,b-a-1}T_1+B(a+1,b-a-1)A(a+1,b-a-1)\\
&=0_{a+1}+I_{a+1}=I_{a+1}.
\end{align*}
This proves that $U \circ T=I_{D_{m+n-a-1}}$.

When $m-a \leq j \leq m$,
\begin{align*}
T \circ U_{|\text{Im }T}(T(v_{j,m+n-a-j}))&=T(U\circ T(v_{j,m+n-a-j}))\\
&=T \circ I_{D_{m+n-a-1}}(v_{j,m+n-a-j})\\
&=T(v_{j,m+n-a-j}).
\end{align*}
Since $(T(v_{m-a,n}),T(v_{m-a+1,n-1}),\dots,T(v_{m,n-a}))$ is a basis for $\text{Im }T$, it follows that $T \circ U_{|\text{Im }T}=I_{\text{Im }T}$.

\end{proof}

\section{Proofs}\label{Proofs}
Before we prove Theorem~\ref{MainTheorem}, we need a lemma.  Recall that $a_i,b_i, \lambda_i$ etc. were defined in Theorem~\ref{MainTheorem}.
\begin{lemma}\label{StoS+1}
Suppose that $a_i+1 \leq i <b_i$.  If 
\[y_i=\sum_{j=m-i+1}^{m} \alpha_{j}v_{j,m+n-i+1-j}\in D_{m+n-i}\]
satisfies $(g-1)^{\lambda_i-1}(y_i)=(g-1)^{m+n-a_i-b_i-1}(y_i)=x_{a_i+b_i+1-i}$, then
\[(g-1)^{\lambda_i-1}\left(\sum_{j=m-i}^{m-1} \alpha_{j+1}  v_{j,m+n-i-j}\right)=-x_{a_i+b_i-i}.\]
\end{lemma}

\begin{proof}
This proof was suggested by the anonymous reviewer and draws on ideas in the proof of Theorem 2.2.2 of I\&I.  Let $h_1$ and $h_2$ be the linear transformations of $V_m \otimes V_n$ into itself determined by $h_1(v_{i,j})=v_{i-1,j}$ and $h_2(v_{i,j})=v_{i,j-1}$ for all $i$ and $j$.
Then $g-1=h_1+h_2$ and $h_1$, $h_2$, and $g-1$ all commute.  Also, when $2 \leq i \leq m$, $h_1(x_i)=-x_{i-1}$.  Thus
\begin{align*}
(g-1)^{\lambda_i-1}\left(\sum_{j=m-i}^{m-1} \alpha_{j+1}  v_{j,m+n-i-j}\right)&=
(g-1)^{\lambda_i-1}(h_1(y_i))\\
&=h_1\left((g-1)^{\lambda_i-1}(y_i)\right)\\
&=h_1(x_{a_i+b_i+1-i})\\
&=-x_{a_i+b_i-i}.
\end{align*}
\end{proof}
\begin{proof}[Proof of Theorem~\ref{MainTheorem}]
It is clear from the proof of Corollary~\ref{CorHou} that once we  identify, for each $i \in [m]$, $y_i \in D_{m+n-i}$ satisfying $(g-1)^{\lambda_i-1}(y_i)=x_{m+n-i-\lambda_i+1}$, then the conclusions $F G y_i \cong V_{\lambda_i}$ and $V_m \otimes V_n =\bigoplus_{i=1}^m F G y_i$ follow.

As in the statement, $i \in [m]$, $a_i+1=\min\{k \mid \lambda_k=\lambda_i\}$, and $b_i=\max\{k \mid \lambda_k=\lambda_i\}$.  Then by~\cite[Proposition 2]{B2015},
\[\lambda_i=\frac{1}{b_i-a_i}\sum_{k=a_i+1}^{b_i} (m+n-2 k+1)=n+(m-b_i)-a_i=m+n-a_i-b_i.
\]
Now by Lemma~\ref{LTMatrix}, the $k \times k$ matrix $A_k$ of $S_k=(g-1)^{m+n-2k} \colon D_{m+n-k} \to D_k$ with respect to $\mathcal{B}_{m+n-k}$ and $\mathcal{B}_k$  has $(i,j)$-entry $\binom{m+n-2k}{m-k+j-i}$.  (Here $1 \leq k \leq m$.)

By the formulations of ~\cite[Theorem 2.2.9]{IandI2009} in~\cite[Propositions 9 and 10]{GPX2}, the integers $k \in [1,m]$ such that $\det A_k \neq 0$ are precisely the $b_i$.  So $S_{b_i}$ is a bijection.

When $a_i+1=b_i$, $i=b_i$, $\lambda_i=m+n-2i+1=m+n-2 b_i+1$, and in this case $S_{b_i}$ maps $D_{m+n-b_i}$ bijectively onto $D_{b_i}$.  Let $y_{b_i}=S_{b_i}^{-1}(x_{b_i})\in D_{m+n-b_i}$.  Then $(g-1)^{\lambda_i-1}(y_{b_i})=x_{b_i}=x_{m+n+1-i-\lambda_i}$. 

Assume $a_i+1<b_i$ and that $i=a_i+1$.  Then $(g-1)^{\lambda_i-1}\colon D_{m+n-a_i-1} \to D_{b_i}$ is the composition  of the injective linear transformation $T=(g-1)^{b-a_i-1}_{|D_{m+n-a_i-1}}$ and the isomorphism $S_{b_i}$.   By Lemma~\ref{Hou}, $x_{b_i} \in \text{Im }S_{b_i} \circ T$ and thus $S_{b_i}^{-1}(x_{b_i}) \in \text{Im }T$. Let $y_{a_i+1}=U(S^{-1}(x_{b_i}))\in D_{m+n-a_i-1}$, where $U$ is defined in Lemma~\ref{TInverse}.  Then
\begin{align*}
(g-1)^{\lambda_i-1}(y_{a_i+1})&=S_{b_i}\circ T(y_{a_i+1})\\
&=S_{b_i} \circ T(U(S_{b_i}^{-1}(x_{b_i})))\\
&=S_{b_i} \circ I_{|\text{Im }T}(S_{b_i}^{-1}(x_{b_i}))\\
&=x_{b_i}\\
&=x_{m+n+1-a_i-1-\lambda_{a_i+1}}.
\end{align*}

Assume that $a_i+1<b_i$ and that $a_i+2 \leq i \leq b_i$.  Then $(g-1)^{\lambda_i-1}$ maps $D_{m+n-i}$ to  $D_{m+n-i-\lambda_i+1}$.
Suppose that
\[y_{a_i+1}=\sum_{j=m-a_i}^m\alpha_{j} v_{j,m+n-a_i-j}\]
is the expression of $y_{a_i+1}$, which we have identified above, in terms of $\mathcal{B}_{m+n-a_i-1}$.  Then
\[(g-1)^{\lambda_{a_i+1}-1}(y_{a_i+1})=(g-1)^{m+n-a_i-b_i-1}(y_{a_i+1})=x_{b_i}=x_{m+n+1-(a_i+1)-\lambda_{a_i+1}}.\]
Recalling our definition of $h_1$ in Lemma~\ref{StoS+1},  let $y_i=(-1)^{i-a_i-1}h_1^{i-a_i-1}(y_{a_i+1})$.  Then by repeated applications of Lemma~\ref{StoS+1}
\begin{align*}
(g-1)^{\lambda_i-1}(y_i)&=(g-1)^{\lambda_i-1}\left((-1)^{i-a_i-1}h_1^{i-a_i-1}(y_{a_i+1})\right)\\
&=(-1)^{i-a_i-1}h_1^{i-a_i-1}\left((g-1)^{\lambda_i-1}(y_{a_i+1})\right)\\
&=(-1)^{i-a_i-1}h_1^{i-a_i-1}\left(x_{m+n+1-(a_i+1)-\lambda_{a_i+1}}\right)\\
&=x_{m+n+1-(a_i+1)-\lambda_{a_i+1}-i+a_i+1}\\
&=x_{m+n+1-i-\lambda_i}.
\end{align*}
\end{proof}

\section{An Example}\label{Example}

When $\text{char }F=7$,
\[V_{12} \otimes V_{13} \cong 4 \cdot V_{21} \oplus V_{16} \oplus V_{14} \oplus V_{12} \oplus 4 \cdot V_7 \oplus V_2.\]

So $\lambda_5=16$ occurs with multiplicity one.  We use Case (1) of Theorem~\ref{MainTheorem} to identify $y_5 \in D_{20}$ such that $(g-1)^{15}(y_5)=x_5$.  Now $S_5=(g-1)^{15} \colon D_{20} \to D_5$ is invertible by~Theorem~\ref{MainTheorem}.   By Lemma~\ref{LTMatrix}, the $5 \times 5$ matrix $A_5$ of $S_5$ with respect to $\mathcal{B}_{20}$ and $\mathcal{B}_5$ has $(i,j)$-entry $\binom{15}{7+j-i}$.

I\@I look for $w \in \text{ker }(g-1)^{16} \colon D_{20} \to D_4$.  It is unique, up to a scalar since $(g-1)^{15} \colon D_{20} \to D_5$ is an isomorphism.

Applying Proposition~\ref{Adjoint}, the matrix of $S_5^{-1}$ with respect to $\mathcal{B}_5$ and $\mathcal{B}_{20}$ is
\[A_5^{-1}=
\begin{pmatrix}
4 & 3 & 4 & 3 & 4\\
0 & 4 & 3 & 4 & 3\\
0 & 0 & 4 & 3 & 4\\
0 & 0 & 0 & 4 & 3\\
0 & 0 & 0 & 0 & 4
\end{pmatrix}
\]
and since
\[A_5^{-1}
\begin{pmatrix}
1 & 6 & 1 & 6 & 1
\end{pmatrix}^T
=
\begin{pmatrix}
6 & 5 & 5 & 6 & 4
\end{pmatrix}^T,
\]
$y_5=S_5^{-1}(x_5)=6 v_{8,13}+5 v_{9,12}+5 v_{10,11}+6 v_{11,10}+4 v_{12,9}$.

Next we focus on $\lambda_8=\lambda_9=\lambda_{10}=\lambda_{11}=7$.  First we use Case (2) of Theorem~\ref{MainTheorem} to identify $y_8 \in D_{17}$ such that $(g-1)^6(y_8)=x_{11}$.  By Theorem~\ref{MainTheorem}, $S_{11}=(g-1)^3 \colon D_{14} \to D_{11}$ is invertible.  By Lemma~\ref{LTMatrix}, the $11 \times 11$  matrix $A_{11}$ of $S_{11}$ with respect to $\mathcal{B}_{14}$ and $\mathcal{B}_{11}$ has $(i,j)$-entry $\binom{3}{1+j-i}$.

Applying  Proposition~\ref{Adjoint}, the matrix of $S_{11}^{-1}$ with respect to $\mathcal{B}_{11}$ and $\mathcal{B}_{14}$ is 
\[\setcounter{MaxMatrixCols}{11}
A_{11}^{-1}=
\begin{pmatrix}
3 & 6 & 1 & 4 & 0 & 1 & 0 & 4 & 1 & 6 & 3\\
1 & 4 & 0 & 1 & 0 & 4 & 1 & 6 & 3 & 0 & 6\\
3 & 5 & 4 & 6 & 0 & 6 & 4 & 5 & 2 & 3 & 1\\
6 & 3 & 1 & 3 & 0 & 4 & 6 & 5 & 5 & 6 & 4\\
0 & 0 & 0 & 0 & 0 & 1 & 4 & 6 & 4 & 1 & 0\\
0 & 0 & 0 & 0 & 0 & 0 & 1 & 4 &  6 & 4 & 1\\
1 & 4 & 6 & 4 & 1 & 0 & 0 & 0 & 0 & 0 & 0\\
4 & 2 & 3 & 2 & 4 & 0 & 0 & 3 & 6 & 1 & 4\\
6 & 3 & 1 & 3 & 6 & 0 & 0 & 1 & 4 & 0 & 1\\
4 & 2 & 3 & 2 & 4 & 0 & 0 & 3 & 5 & 4 & 6\\
1 & 4 & 6 & 4 & 1 & 0 & 0 & 6 & 3 & 1 & 3
\end{pmatrix}.
\]
Since
\begin{align*}
A_{11}^{-1}&
\begin{pmatrix}
1 & 6 & 1 & 6 & 1 & 6 & 1 & 6 & 1 & 6 & 1
\end{pmatrix}^T\\
&=
\begin{pmatrix}
1 & 3 & 3 & 1 & 0 & 0 & 0 & 6 & 4 & 4 & 6
\end{pmatrix}^T,
\end{align*}
$S_{11}^{-1}(x_{11})=v_{2,13}+3 v_{3,12}+3 v_{4,11}+v_{5,10}+6 v_{9,6}+4 v_{10,5}+4 v_{11,4}+6 v_{12,3}$.
We can use either the formula for $U$ in the statement or the matrix of $U$ in the proof of Lemma~\ref{TInverse} to show that $U(S_{11}^{-1}(x_{11}))=v_{5,13}+6 v_{12,6}$.
Thus $y_8=v_{5,13}+6 v_{12,6}$.

Starting with $y_8=v_{5,13}+6 v_{12,6}$ and applying Case (3) of Theorem~\ref{MainTheorem} three times, we set $y_{9}=6v_{4,13}+v_{11,6}$, $y_{10}=v_{3,13}+6v_{10,6}$, and $y_{11}=6v_{2,13}+v_{9,6}$.  It follows that 
$(g-1)^6(y_9)=x_{10}$, $(g-1)^6(y_{10})=x_9$, and $(g-1)^6(y_{11})=x_8$, as required.

\textit{Acknowledgement: }We thank an anonymous reviewer for helping us realize  the connection between our work and that  of I\&I, for many helpful suggestions including the proof of Lemma~\ref{StoS+1},  and lastly for bringing  Nordenstam and Young's paper to our attention.

\end{document}